\documentclass[reqno,centertags, 12pt]{amsart}
\usepackage{amsmath,amsthm,amscd,amssymb}
\usepackage{latexsym}
\usepackage{graphicx}
\usepackage{enumitem}
\usepackage[mathscr]{eucal}
\newcommand{\bbC}{{\mathbb{C}}}
\newcommand{\bbD}{{\mathbb{D}}}

\newcommand{\bbK}{{\mathbb{K}}}

\newcommand{\bbR}{{\mathbb{R}}}

\newcommand{\bbZ}{{\mathbb{Z}}}

\newcommand{\fre}{{\mathfrak{e}}}

\newcommand{\frA}{{\mathfrak{A}}}
\newcommand{\x}{{\mathbf{x}}}

\newcommand{\calC}{{\mathcal{C}}}

\newcommand{\calG}{{\mathcal G}}
\newcommand{\calH}{{\mathcal H}}

\newcommand{\bdone}{{\boldsymbol{1}}}

\newcommand{\lb}{\label}

\newcommand{\wti}{\widetilde  }

\newcommand{\bi}{\bibitem}

\newcommand{\beq}{\begin{equation}}
\newcommand{\eeq}{\end{equation}}
\newcommand{\ba}{\begin{align}}
\newcommand{\ea}{\end{align}}

\newcounter{smalllist}

\newcommand{\comm}[1]{}

\allowdisplaybreaks
\numberwithin{equation}{section}

\newtheorem{theorem}{Theorem}[section]

\newtheorem*{p2.1}{Proposition 2.1}
\newtheorem{proposition}[theorem]{Proposition}
\newtheorem{lemma}[theorem]{Lemma}

\theoremstyle{definition}
\newtheorem*{remark}{Remark}
\newtheorem*{remarks}{Remarks}

\newcommand{\jap}[1]{\langle #1 \rangle}
\newcommand{\norm}[1]{\lVert#1\rVert}

\begin{document}

\title[Chebyshev Polynomials, II]{Asymptotics of Chebyshev Polynomials,\\ II. DCT Subsets of $\bbR$}
\author[J.~S.~Christiansen, B.~Simon, P.~Yuditskii and M.~Zinchenko]{Jacob S.~Christiansen$^{1,5}$, Barry Simon$^{2,6}$, Peter Yuditskii$^{3,7}$\\and
Maxim~Zinchenko$^{4,8}$}

\thanks{$^1$ Centre for Mathematical Sciences, Lund University, Box 118, 22100 Lund, Sweden.
 E-mail: stordal@maths.lth.se}

\thanks{$^2$ Departments of Mathematics and Physics, Mathematics 253-37, California Institute of Technology, Pasadena, CA 91125.
E-mail: bsimon@caltech.edu}

\thanks{$^3$ Institute for Analysis, Johannes Kepler University Linz, 4040 Linz, Austria
E-mail: Petro.Yudytskiy@jku.at}

\thanks{$^4$ Department of Mathematics and Statistics, University of New Mexico,
Albuquerque, NM 87131, USA; E-mail: maxim@math.unm.edu}

\thanks{$^5$ Research supported in part by Project Grant DFF-4181-00502 from the Danish Council for Independent Research.}

\thanks{$^6$ Research supported in part by NSF grants DMS-1265592 and DMS-1665526 and in part by Israeli BSF Grant No. 2014337.}

\thanks{$^7$ Research supported by the Austrian Science Fund FWF, project no: P29363-N32.}

\thanks{$^8$ Research supported in part by Simons Foundation grant CGM--281971.}

\

\date{\today}

\begin{abstract}  We prove Szeg\H{o}--Widom asymptotics for the Chebyshev polynomials of a compact subset of $\bbR$ which is regular for potential theory and obeys the Parreau--Widom and DCT conditions.
\end{abstract}

\maketitle

\section{Introduction} \lb{s1}

Let $\fre \subset \bbR$ be a compact subset with logarithmic capacity $C(\fre) > 0$.  Define
\begin{equation}\label{1.1}
  \norm{f}_\fre = \sup_{x \in \fre} |f(x)|
\end{equation}
The Chebyshev polynomial, $T_n(z)$, is the monic polynomial with
\begin{equation}\label{1.2}
  t_n \equiv \norm{T_n}_\fre = \inf \{\norm{P}_\fre \, | \, \deg P = n, P \textrm{ monic}\}
\end{equation}
It is a consequence of the alternation theorem (a result of Borel \cite{Borel} and Markov \cite{Markov} using ideas that go back to Chebyshev; see \cite{CSZ1} for a statement and proof) that $T_n$ is unique and that
\begin{equation}\label{1.3}
  \fre_n \equiv T_n^{-1}([-t_n,t_n]) = \{z \in \bbC \,|\, -t_n \le T_n(z) \le t_n\}
\end{equation}
is a subset of $\bbR$.  Clearly, by definition of $t_n$,
\begin{equation}\label{1.4}
  \fre \subset \fre_n
\end{equation}

Recall that the Green's function, $G_\fre(z)$, is the unique function on $\bbC$ which is positive and harmonic on $\bbC \setminus \fre$, upper semicontinuous on $\bbC$, so that $G_\fre(z) = \log(|z|) + \textrm{O}(1)$ near $z = \infty$ and so that $G_\fre(x) = 0$ for quasi-every $x \in \fre$.  A set, $\fre$, is called regular (for potential theory) if $G_\fre(x) = 0$ for all $x \in \fre$ (which implies that $G_\fre$ is continuous on $\bbC$).  We'll assume that $\fre$ is regular.  One has that near infinity
\begin{equation}\label{1.5}
  G_\fre(z) = \log(|z|) - \log(C(\fre)) + \textrm{O}(1/|z|)
\end{equation}
Moreover, if $d\rho_\fre$ is the potential theoretic equilibrium measure for $\fre$, then
\begin{equation}\label{1.6}
  G_\fre(z) = -\log(C(\fre)) + \int \log(|z-x|) d\rho_\fre(x)
\end{equation}
For more on potential theory, see \cite[Section 3.6]{HA}.

It is not hard to see (see \cite{CSZ1}) that the Green's function, $G_n$, for $\fre_n$ is
\begin{equation}\label{1.5a}
G_n(z) = \frac{1}{n} \log\left(\left|\frac{T_n(z)}{t_n}+i\sqrt{1-\left(\frac{T_n(z)}{t_n}\right)^2}\right|\right)
\end{equation}
which implies that
\begin{equation}\label{1.6a}
  t_n = 2(C(\fre_n))^n
\end{equation}
In particular, since $C(\fre) \le C(\fre_n)$, we get Schiefermayr's bound \cite{Schie}
\begin{equation}\label{1.7}
  t_n \ge 2(C(\fre))^n
\end{equation}

In \cite{CSZ1}, we introduced the term \emph{Totik--Widom bound} (after \cite{Totik09, Widom}) if for some constant $D$, one has that
\begin{equation}\label{1.8}
  t_n \le D(C(\fre))^n
\end{equation}

A compact set $\fre \subset \bbC$ is said to obey a Parreau--Widom (PW) condition (after \cite{Parreau, Widom71}) if and only if
\begin{equation}\label{1.9}
  PW(\fre) \equiv \sum_{z_j \in \calC} G_\fre(z_j) < \infty
\end{equation}
where $\calC$ is the set of points, $z_j$, where $\nabla G_\fre(z_j) =0$.  For regular subsets of $\bbR$, all these critical points are real and there is exactly one such point in each bounded open component, $K_j$, of $\bbR \setminus \fre$ and $G_\fre(z_j) = \max_{x \in K_j} G_\fre(x)$.

In \cite{CSZ1}, we proved that if $\fre \subset \bbR$ is a regular PW set, then one has an explicit Totik--Widom bound
\begin{equation}\label{1.10}
  t_n \le 2 \exp{(PW(\fre))} (C(\fre))^n
\end{equation}

Our methods there say nothing about the complex case.  In this regard, we mention the recent interesting paper of Andrievskii \cite{Andr} who has proven Totik--Widom bounds for a class of sets that, for example, includes the Koch snowflake.

One of our results in this paper (see Theorem \ref{T1.4} and Section \ref{s2}) will be a kind of weak converse -- under an additional condition on $\fre$ which should hold generically, if $\fre \subset \bbC$ is compact, regular and obeys a Totik--Widom bound, then $\fre$ is a PW set.

For a general positive capacity, regular, compact set $\fre \subset \bbC$, we define $\Omega$ to be its complement in the Riemann sphere, i.e.,
\begin{equation}\label{1.11}
  \Omega = (\bbC \cup \{\infty\}) \setminus \fre
\end{equation}
which we suppose is connected (this always holds if $\fre \subset \bbR$).  We let $\wti{\Omega}$ be its universal cover and $\pi: \wti{\Omega} \to \Omega$ the covering map.  It is a consequence of the uniformization theorem (see \cite[Section 8.7]{BCA}) that $\wti{\Omega}$ is conformally equivalent to the disk, $\bbD$, a fact we will use.  We denote by $\x: \bbD \to \Omega$ the unique covering map normalized by $\x(0) = \infty$ and near $z=0$, $\x(z) = Dz^{-1}+\textrm{O}(1)$ with $D>0$.

There is an important multivalued analytic function, $B_\fre(z)$, on $\Omega$ determined by
\begin{equation}\label{1.12}
  |B_\fre(z)| = e^{-G_\fre(z)}
\end{equation}
and that near $\infty$,
\begin{equation}\label{1.13}
  B_\fre(z) = C(\fre)z^{-1} + \textrm{O}(z^{-2})
\end{equation}
One way of constructing it is to use the fact that $-G_\fre$ has a harmonic conjugate locally so that locally on $\bbC \setminus \fre$, it is the real part of an analytic function whose exponential is $B_\fre(z)$.  It is easy to see that this allows $B_\fre$ to be continued along any curve in $\wti{\Omega}$ so by the monodromy theorem (\cite[Section 11.2]{BCA}), $B_\fre(z)$ has an analytic continuation to $\wti{\Omega}$ which defines a multivalued analytic function on $\Omega$.

By analyticity, \eqref{1.12} holds for all branches of $B_\fre(z)$.  In particular, going around a closed curve, $\gamma$, can only change $B_\fre$ by a phase factor which implies there is a character, $\chi_\fre$, of the fundamental group, $\pi_1(\Omega)$, so that going around $\gamma$ changes $B_\fre$ by $\chi_\fre([\gamma])$.  It is not hard to see (\cite[Theorem 2.7]{CSZ1}) that
\begin{equation}\label{1.14}
  \chi_\fre(\gamma) = \exp\left(-2\pi i \int_\fre N(\gamma,x) d\rho_\fre(x)\right)
\end{equation}
where $N(\gamma,x)$ is the winding number for the curve $\gamma$ about $x$.  Thus $B_\fre$ is a character automorphic function.

An alternate construction is to consider elementary Blaschke factors $b(z,w) (=(\bar{w}/|w|)[(w-z)/(1-\bar{w}z)]$ if $w \ne 0$) for $z,w \in \bbD$.  Then, lifted to $\bbD$,
\begin{equation}\label{1.15}
  B_\fre(z) = \prod_{\{w_j\,|\,\x(w_j)=\infty\}} b(z,w_j)
\end{equation}
We will call $B_\fre$ the \emph{canonical Blaschke product for} $\fre$ and $\chi_\fre$, \emph{the canonical character}.

Similarly, we can define for each $w \in \Omega$, $B_\fre(z,w)$ either by using \eqref{1.15} with $\{w_j\,|\,\x(w_j)=\infty\}$ replaced by $\{w_j\,|\,\x(w_j)=w\}$ or by using the Green's function $G_\fre(z,w)$ with pole at $w$ and demanding that $|B_\fre(z,w)| = \exp(-G_\fre(z,w))$ and fixing the phase by demanding that $B_\fre(\infty,w) > 0$.

One can consider character automorphic functions for general characters, $\chi \in \pi_1(\Omega)^*$, the full character group.  In this regard the following theorem of Widom \cite{Widom71} (see also Hasumi \cite[Theorem 5.2B]{Hasu}) is important:

\begin{theorem} \lb{T1.1} (Widom) Suppose that $\fre$ is a compact set regular for potential theory. Then $\fre$ is a PW set if and only if for every character, $\chi \in \pi_1(\Omega)^*$, there is a non-zero analytic $\chi$-automorphic function on $\wti{\Omega}$ which is bounded.
\end{theorem}

Single-valued analytic functions on $\wti{\Omega}$ correspond to multi-valued functions on $\Omega$ and we will often refer to them as if they are ordinary functions.  In essence we view $\Omega$ with the convex hull of $\fre$ removed as a subset of $\wti\Omega$.

For a PW set, $\fre$, and any character, $\chi$, we let $H^\infty(\Omega,\chi)$ be the set of bounded analytic $\chi$-automorphic functions on $\wti\Omega$ and denote by $\|\cdot\|_\infty$ the corresponding norm.  We use $H^2(\Omega,\chi)$ or $\calH_\chi$ for the set of analytic $\chi$-automorphic functions, $f$, for which $|f|^2$ has a harmonic majorant in $\Omega$. Evidently, $H^\infty(\Omega,\chi)\subset H^2(\Omega,\chi)$. It is easy to see that $H^2(\Omega,\chi)$ is precisely those $\chi$-automorphic functions, $f$, on $\Omega$ whose lifts to $\bbD$ under $\x$ are in $H^2(\bbD)$.

When $\fre$ is a PW set, there exist $h \in H^\infty(\Omega,\chi)$ with $h(\infty) \ne 0$, for if $f \in H^\infty(\Omega,\chi)$ with $f(z) = Cz^{-n}+\textrm{O}(z^{-n-1}); \, C \ne 0$, then $h(z) = z^n f(z)$ is also in $H^\infty(\Omega,\chi)$ and $h(\infty) = C$.

For any $\chi$, the Widom trial functions for $\chi$ is the set, $\{h \in H^\infty(\Omega,\chi) \, | \, h(\infty)=1\}$.  The \emph{Widom minimizer}, $F_\chi(z)$, is a bounded $\chi$-character automorphic function with $F_\chi(\infty)=1$ so that
\begin{equation}\label{1.15A}
  \norm{F_\chi}_\infty = \inf\{\norm{h}_\infty \, | \, h \in H^\infty(\Omega,\chi); \, h(\infty) = 1\}
\end{equation}
Knowing that there are Widom trial functions, it is easy to prove using Montel's Theorem (\cite[Section 6.2]{BCA}) that minimizers exist.  In Section \ref{s2}, we'll prove that minimizers are unique (this is not a new result although our proof is simpler than previous ones).

We will also consider a dual problem.  The dual Widom trial functions are $\{g \in H^\infty(\Omega,\chi) \, | \, \norm{g}_\infty = 1\}$.  The \emph{dual Widom maximizer} is that function $Q_\chi$ in the dual Widom trial functions with
\begin{equation}\label{1.15B}
  Q_\chi(\infty) = \sup\{g(\infty )\, | \, g \in H^\infty(\Omega,\chi), \, \norm{g}_\infty = 1, \, g(\infty) > 0\}
\end{equation}
If $g$ is a dual Widom trial function with $g(\infty) \ne 0$, then $g/g(\infty)$ is a Widom trial function.  Conversely, if $h$ is a Widom trial function, then $h/\norm{h}_\infty$ is a dual Widom trial function.  This shows that for the two problems, either both or neither have unique solutions and
\begin{equation}\label{1.16}
  Q_\chi = F_\chi/\norm{F_\chi}_\infty, \quad F_\chi=Q_\chi/Q_\chi(\infty), \quad Q_\chi(\infty) = 1/\norm{F_\chi}_\infty
\end{equation}

Suppose now that $\fre \subset \bbC$ is compact, connected and simply connected.  Then $\Omega$ is simply connected and $B_\fre$ is analytic (rather than multivalued analytic) and is, in fact, the Riemann map of $\Omega$ to $\bbD$ (uniquely specified by $B_\fre(\infty) = 0$ and that near $\infty$, $B_\fre(z) = Cz^{-1} + \textrm{O}(z^{-2})$ with $C>0$).  In 1919, assuming that $\partial\Omega$ is an analytic Jordan curve, Faber \cite{Faber} proved that in this case
\begin{equation}\label{1.17}
  \frac{T_n(z) B_\fre(z)^n}{C(\fre)^n} \to 1
\end{equation}
uniformly on $\overline{\Omega}$.

In 1969, Widom \cite{Widom} considered $\fre \subset \bbC$ which is a finite union of $C^{1+}$ Jordan curves and arcs.  He noted that \eqref{1.17} couldn't
hold when there was more than one arc or curve since, in that case, $B_\fre(z) ^n$ is now a character automorphic function with character
$\chi_\fre^n$. If $F_n \equiv F_{\chi_\fre^n}$, Widom suggested what we call the \emph{Widom surmise}, that
\begin{equation}\label{1.18}
  \frac{T_n(z) B_\fre(z)^n}{C(\fre)^{n}} -F_n(z) \to 0
\end{equation}
uniformly on compact subsets of $\wti{\Omega}$. He proved this when $\fre$ consisted only of (closed) Jordan curves and in \cite{CSZ1}, we proved it for $\fre$ a finite gap set in $\bbR$.

We say that $T_n$ has \emph{strong Szeg\H{o}--Widom asymptotics} if (see \cite[Section 6.6]{OT} for a discussion of almost periodic functions)
\vspace{-0.08cm}
\begin{enumerate}[label=(\alph*)]
  \item \eqref{1.18} holds uniformly on compact subsets of $\wti{\Omega}$
  \item $n \mapsto \norm{F_n}_\infty$ is an almost periodic function
  \item $n \mapsto F_n(z)$ is an almost periodic function uniformly on compact subsets of $\wti{\Omega}$.
\end{enumerate}
\vspace{-0.08cm}
We note that the above results of Widom \cite{Widom} and \cite{CSZ1} prove (b) and (c) also.

A final element we need before stating our main theorem is the notion of the Direct Cauchy Theorem (DCT) property.  There are many equivalent definitions of DCT -- see Hasumi \cite{Hasu} or Volberg--Yuditskii \cite{VY}.  Rather than stating a formal definition, we first of all quote a theorem that could be used as one definition of DCT:

\begin{theorem} [Hayashi \cite{Hay}, Hasumi \cite{Hasu}] \lb{T1.2}
A PW set $\fre$ obeys a DCT if and only if the function $\chi \mapsto Q_\chi(\infty)$ of the dual Widom maximizer problem is a continuous function on $\pi_1(\Omega)^*$.
\end{theorem}

We'll also quote as needed some other results that rely on the DCT condition.  We note that any homogeneous subset of $\bbR$ (in the sense of Carleson \cite{SY}) obeys DCT \cite{SY}.  On the other hand, Hasumi \cite{Hasu} has found rather simple explicit examples (with thin components) of subsets of $\bbR$ which obey PW but not DCT.  Volberg--Yuditskii \cite{VY} have even found examples all of whose reflectionless measures are absolutely continuous.

We can now state the main result of this paper:

\begin{theorem} \lb{T1.3} Let $\fre \subset \bbR$ be a compact set which is regular for potential theory and that obeys the PW and DCT conditions.  Then its Chebyshev polynomials have strong Szeg\H{o}--Widom asymptotics.  Moreover,
\begin{equation}\label{1.18A}
  \lim_{n \to \infty} \frac{t_n}{C(\fre)^n \norm{F_n}_\infty} = 2
\end{equation}
\end{theorem}

\begin{remarks} 1. Given the limit \eqref{1.18}, the $2$ in \eqref{1.18A} may seem surprising.  Widom noted the $2$ in the easy special case $\fre = [-1,1]$ and proved \eqref{1.18A} for general finite gap subsets of $\bbR$.  This fact was used in our proof of \eqref{1.18} for the finite gap case in \cite{CSZ1}.  Here we'll prove \eqref{1.18} first and then prove \eqref{1.18A}.

2.  Our proof uses a partially variant strategy to the one in \cite{CSZ1} and we believe is simpler even in the finite gap case (especially if you include the need there for some results of Widom that we don't need to prove a priori).
\end{remarks}

For our other main results, we need a new definition.  We say a set $\fre \subset \bbR$ has a \emph{canonical generator} if $\{\chi_\fre^n\}_{n=-\infty}^\infty$ is dense in the character group $\pi_1(\Omega)^*$.  This holds if and only if for each  decomposition $\fre = \fre_1 \cup \dots \cup \fre_\ell$ into closed disjoint sets and rational numbers $\{q_j\}_{j=1}^{\ell-1}$, we have that
\begin{equation}
  \label{independent}
  \sum_{j=1}^{\ell-1} q_j\rho_\fre(\fre_j) \ne 0
\end{equation}

\begin{remarks} 1. The class of regular PW sets can be parametrized by comb domains of the form
\begin{equation}
  \label{Pi}
  \Pi=\{ x+iy \, | \, 0<x<1, y>0 \} \setminus \cup_k \{ \omega_k+iy \, | \, 0<y\leq h_k\}
\end{equation}
with $\omega_k\in(0,1)$, $\omega_k\neq\omega_j$ for $k\neq j$ and $h_k>0$, $\sum_k h_k<\infty$. Specifically, if $\fre$ is scaled to the interval $[0,1]$, then
\begin{equation}
  \label{theta}
  \theta(z)=\frac{-\log B_\fre(z)}{\pi i}
\end{equation}
is a conformal mapping of $\bbC_+$ onto such a domain (see \cite{EY} for more details).
In that parametrization, the property of a canonical generator is generic.
For one can show that $\omega_k=\rho_\fre(\{x\in\fre\,|\,x\leq a_k\})$ and the collection of comb domains with rationally independent $\omega_k$'s clearly form a dense $G_\delta$ set.

2.  It seems likely that the condition of a canonical generator holds in various other generic senses as well. For example, given a fixed nowhere dense, infinite gap set, we can pick a positive integer labeling of the gaps and, for any $\lambda\in\prod_1^\infty [1/2, 2]$, consider the set obtained by scaling the $j$th gap by $\lambda_j$. We suspect the set of $\lambda$'s for which this set has a canonical generator, is a dense $G_\delta$. In the finite gap case, that this is true follows from results of Totik \cite{Totik09}.
\end{remarks}

\begin{theorem} \lb{T1.4}  Let $\fre \subset \bbC$ be a compact set regular for potential theory with a canonical generator.  If $\fre$ has a Totik--Widom bound, then $\fre$ is a PW set.
\end{theorem}

\begin{remarks} 1.  While we need to assume canonical generator, this result suggests that Totik--Widom fails if the set is not PW.

2.  We emphasize that this result holds for $\fre \subset  \bbC$ and not just $\fre \subset \bbR$.
\end{remarks}

\begin{theorem} \lb{T1.5} Let $\fre \subset \bbC$ be a compact set regular for potential theory with a canonical generator.  Suppose that $\fre$ is a PW set and that $n \mapsto \norm{F_n}_\infty$ is a bounded almost periodic function on $\bbZ$.  Then $\fre$ is a DCT set.
\end{theorem}

\begin{remarks}  1.  Again, we emphasize that this holds for all $\fre \subset \bbC$ not just $\fre \subset \bbR$.

2.  So, one small part of Szeg\H{o}--Widom asymptotics, namely asymptotic almost periodicity of $\norm{T_n}_\fre/C(\fre)^n$ and the limit result \eqref{1.18A}, implies that $\fre$ is a DCT set (at least if $\fre$ has a canonical generator).
\end{remarks}

We will note results from \cite{CSZ1} as needed but mention some that are needed to overview the contents of the paper.  Let $B_n \equiv B_{\fre_n}$.  Then \cite{CSZ1} proved that
\begin{equation}\label{1.19}
   \frac{2T_n(z)}{t_n} = B_n(z)^n + B_n(z)^{-n}
\end{equation}
Thus, instead of looking at
\begin{equation}\label{1.20}
  L_n(z) \equiv \frac{T_n(z) B_\fre(z)^n}{C(\fre)^n}
\end{equation}
we'll look at
\begin{equation}\label{1.21}
  M_n(z) = B_\fre(z)^n/B_n(z)^n
\end{equation}
which obeys
\begin{equation}\label{1.22}
  |M_n(z)| = \exp(-nh_n(z)), \qquad h_n(z) \equiv G_\fre(z) - G_{\fre_n}(z)
\end{equation}
By \eqref{1.19}
\begin{equation}\label{1.23}
  L_n(z) = (1+B_n(z)^{2n})H_n(z), \quad H_n(z) = \frac{C(\fre_n)^n}{C(\fre)^n}\frac{B_\fre(z)^n}{B_n(z)^n}=\frac{M_n(z)}{M_n(\infty)}
\end{equation}
The first equation in \eqref{1.23} explains the $2$ in \eqref{1.18A}.  By a simple argument,
\begin{equation}\label{1.23A}
  \sup_{n, z \in K} |B_n(z)| < 1 \textrm{ for any compact set } K \subset \wti{\Omega}
\end{equation}
so that $B_n(z)^{2n}$ goes to zero, but for $\sup_{z \in \Omega} |1+B_n(z)^{2n}|$, we get $2$ since there are points $x \in \fre_n$ with $B_n(x+i0) = 1$.

By the first equation in \eqref{1.23} and \eqref{1.23A}, \eqref{1.18} is equivalent to
\begin{equation}\label{1.24}
  H_n(z) - F_n(z) \to 0
\end{equation}
By the second equation in \eqref{1.23}, it seems likely that it suffices to control limits of $M_n$ and that is what we'll do.  By the maximum principle for harmonic functions and \eqref{1.22}, $|M_n(z)| \le 1$.  We will prove that $\lim_{n \to \infty} \norm{M_n}_\infty = 1$ and that limit points of $M_n$ with $n_j \to \infty$ so that $\chi_\fre^{n_j} \to \chi_0$ for some $\chi_0 \in \pi_1(\Omega)^*$ are dual Widom maximizers which will let us prove \eqref{1.24}.

Here is an overview of the rest of this paper.  In section \ref{s2}, following ideas of Fisher \cite{Fish}, we prove uniqueness of solutions of the Widom minimization problem (this is not a new result -- only a new proof -- see the discussion there) and prove Theorem \ref{T1.4}.  In Section \ref{s3}, we discuss continuity of $F_\chi$ in $\chi$ and prove Theorem \ref{T1.5}.  In Section \ref{s4}, we prove that limit points of the $M_n$ are Blaschke products of suitable $B(z,x_j)$ and in Section \ref{s5} that these products are dual Widom maximizers.  This result has been obtained by Volberg--Yuditskii \cite{VY} but we found an alternate proof using ideas of Eichinger--Yuditskii \cite{EichY}.  Finally, in Section \ref{s6}, we put things together and prove Theorem \ref{T1.3}

\section{Uniqueness of the Dual Widom Maximizer } \lb{s2}

In this section, we provide a proof of uniqueness of solutions of the dual Widom maximizer problem and so uniqueness of solutions of the Widom minimizer problem.  If $\fre$ obeys a PW condition, $H^\infty(\Omega,\chi)$ is non-empty (by Theorem \ref{T1.1}) and so contains $h$ with $h(\infty) > 0$.  By Montel's theorem (\cite[Section 6.2]{BCA}), $\{h \in H^\infty(\Omega,\chi) \,|\, \norm{h}_\infty \le 1, h(\infty) \ge 0\}$ is compact in the topology of uniform convergence on compact subsets of $\wti{\Omega}$.  Thus, there exists a maximizer.  We need to prove that this is unique.

Recall that the Ahlfors problem for a compact set $\fre \subset \bbC$ is to look for bounded analytic functions, $f$, on $\Omega = (\bbC \cup \{\infty\}) \setminus \fre$ with $\sup_{z \in \Omega} |f(z)| \le 1$ and $f(\infty) =0$ that maximize $f'(\infty)$ (defined by $f(z) = f(\infty)+f'(\infty)z^{-1}+\textrm{O}(z^{-2})$ near $z=\infty$).  This maximum is called the analytic capacity (because if ``analytic'' is replaced by ``harmonic'', the maximum is the potential theoretic capacity).  There is an enormous literature on the Ahlfors problem, in particular two sets of lecture notes \cite{Garnett, Pajot} and a textbook presentation in \cite[Section 8.8]{BCA}.

This is clearly analogous to the dual Widom maximizer problem so proofs of uniqueness for the Ahlfors problem should have analogs for our problem.  In his original paper, Ahlfors \cite{Ahl} considered an $n$-connected domain $\Omega$ (i.e., $\fre \subset \bbC$ has $n$ connected components) and proved that any maximizer, $g$, has limiting values for almost every point in $\partial\Omega$ (maybe only one sided if $\fre$ has a one dimensional component) with $|g(w)|=1$ for $w \in \partial\Omega$.  This can be used to prove uniqueness.  In \cite{Widom}, Widom proved that uniqueness for the dual maximizer by proving any maximizer had absolute value one on $\partial\Omega$.  The same idea occurs for general Parreau--Widom sets in Volberg--Yuditskii \cite{VY} who had the first proof of the result in this section.

A simple, elegant approach to uniqueness of the Ahlfors problem is due to Fisher \cite{Fish}.  We will modify his approach to accommodate change of character and the fact that the vanishing at $\infty$ is different.

\begin{theorem} \lb{T2.1} Let $\fre \subset \bbC$ be a PW set regular for potential theory.  Then for any character $\chi \in \pi_1(\Omega)^*$, the dual Widom maximizer (and so also the Widom minimizer) exists and is unique.
\end{theorem}

\begin{remarks} 1.  As noted above this has already been proven by Volberg--Yuditskii \cite{VY} but starting from first principles, our proof is simpler.

2.  Uniqueness implies that the maximizer in the dual problem is an extreme point in $H^\infty(\Omega,\chi)_1$, the closed unit ball in $H^\infty(\Omega,\chi)$.  For if $Q_\chi = \tfrac{1}{2} (q_1+q_2)$ with $q_j \in H^\infty(\Omega,\chi)_1$, then by the maximum property, $q_j(\infty) = Q_\chi(\infty)$. So the $q_j$ are also maximizers, and hence equal to $Q_\chi$.
\end{remarks}

\begin{proof} Without loss, we can suppose $\chi \not\equiv 1$ since if $\chi \equiv 1$, the unique dual maximizer is $f \equiv 1$.  In particular, since $\chi \not\equiv 1$, we have that $f(\infty) < 1$ by the maximum principle.  Let $f_1$ and $f_2$ be two maximizers and define
\begin{equation}\label{2.1}
  f = \tfrac{1}{2}(f_1+f_2), \qquad k = \tfrac{1}{2}(f_1-f_2)
\end{equation}
Pick $q \in H^\infty(\Omega,\overline{\chi})$ with $q(\infty) \ne 0$ and $\norm{q}_\infty = 1$ which exists by the PW condition and Theorem \ref{T1.1}.

Since $\norm{f_j}_\infty=1$, we have that $\norm{f\pm k}_\infty = 1$ so
\begin{equation}\label{2.2}
  |f|^2+|k|^2 = \tfrac{1}{2}\left(|f+k|^2+|f-k|^2\right) \le 1
\end{equation}
Define
\begin{equation}\label{2.3}
  g=q k^2/2
\end{equation}
so $g \in H^\infty(\Omega,\chi)$.  By \eqref{2.2},
\begin{equation*}
  |g| \le \frac{1-|f|^2}{2}=(1-|f|)\left(\frac{1+|f|}{2}\right) \le 1-|f|
\end{equation*}
so
\begin{equation}\label{2.4}
  |g|+|f| \le 1
\end{equation}

Since $f_1(\infty)=f_2(\infty)$ is the maximum value, $g(\infty) = 0$, so if $g \not\equiv 0$, then, near $\infty$, we can write
\begin{equation}\label{2.5}
  g(z) = \sum_{k=\ell}^{\infty}a_kz^{-k}, \qquad a_\ell \ne 0
\end{equation}
for some $\ell \ge 1$.

We'll consider as a trial function
\begin{equation}\label{2.6}
  h_\epsilon(z) = f(z) + \epsilon \bar{a}_\ell z^\ell g(z)
\end{equation}
where $\epsilon$ will be picked below.  Since $f(\infty) \in (0,1)$, we can pick $\epsilon_0 > 0$ so that
\begin{equation}\label{2.7}
  f(\infty)+\epsilon_0 |a_\ell|^2 < 1
\end{equation}
Therefore, we can find $R>0$ so that
\begin{equation}\label{2.8}
  |z|>R \Rightarrow |f(z)| + \epsilon_0 |a_\ell| |z^\ell g(z)| < 1
\end{equation}
Pick $\epsilon_1 > 0$ so that
\begin{equation}\label{2.9}
  \epsilon_1 < \epsilon_0, \qquad \epsilon_1|a_\ell| R^\ell < 1
\end{equation}
We claim that $\norm{h_{\epsilon_1}} \le 1$, for by \eqref{2.8} if $|z| > R$, then $|h_{\epsilon_1}(z)| \le 1$, and, if $|z| \le R$, then by \eqref{2.9}
\begin{equation*}
  |h_{\epsilon_1}(z)| \le |f(z)| + \epsilon_1|a_\ell|R^\ell |g(z)| < |f(z)|+|g(z)| \le 1
\end{equation*}
by \eqref{2.4}.  Thus $h_{\epsilon_1}$ is a trial function for the dual Widom problem.

On the other hand,
\begin{equation}\label{2.10}
  h_{\epsilon_1}(\infty) = f(\infty) + \epsilon_1|a_\ell|^2 > f(\infty)
\end{equation}
violating maximality.  We conclude that $g \equiv 0$, so $k \equiv 0$, and $f_1=f_2$.
\end{proof}

\begin{proof} [Proof of Theorem 1.4]  Suppose we have a Totik--Widom bound
\begin{equation}\label{2.11}
  t_n \le D(C(\fre))^n
\end{equation}
Given $\chi_\infty \in \pi_1(\Omega)^*$, pick $n_j \to \infty$ so that $\chi_\fre^{n_j}$, the character of $B_\fre^{n_j}$, converges to $\chi_\infty$ (which we can do by the assumption of canonical generator).  Let
\begin{equation}\label{2.12}
  f_j(z) = \frac{T_{n_j}(z)B_\fre(z)^{n_j}}{C(\fre)^{n_j}}
\end{equation}

By the maximum principle,
\begin{equation*}
     \norm{f_j}_\infty \le \sup_{z \to \fre}|f_j(z)| \le t_{n_j}C(\fre)^{-n_j} \le D
\end{equation*}
so by Montel's theorem, we can find $j_k \to \infty$, so that $f_{j_k}$ converges to $f_\infty$ uniformly on compacts.  Since $T_{n_j}$ is monic and $B_\fre(z)=C(\fre)/z + \textrm{O}(z^{-2})$, we have $f_j(\infty) = 1$ and, therefore, $f_\infty$ is non-zero.  Clearly, $f_\infty \in H^\infty(\Omega,\chi_\infty)$.  By Theorem \ref{T1.1}, $\fre$ obeys a PW condition.
\end{proof}

\section{Continuity of the Widom Minimizer} \lb{s3}

In this section, we study continuity properties (in $\chi$) of $Q_\chi(z)$, $F_\chi(z)$ and $\norm{F_\chi}_\infty$.  We'll show there is continuity if and only if the DCT holds.  Applying this to $n \to F_{\chi_\fre^n}$, we'll see that DCT implies almost periodicity.

\begin{theorem} \lb{T3.1} Let $\fre \subset \bbC$ be a compact, PW and DCT set that is regular for potential theory.  Then $\chi \mapsto Q_\chi$ and $\chi \mapsto F_\chi$ are continuous in the topology of uniform convergence on compact subsets of $\wti{\Omega}$.  Moreover, $\chi \mapsto \norm{F_\chi}_\infty$ is continuous.  Conversely, if $\chi \mapsto \norm{F_\chi}_\infty$ is continuous for $\fre$ a regular PW set, then $\fre$ is a DCT set.
\end{theorem}

\begin{proof} By Theorem \ref{T1.2}, if $\fre$ is a DCT set, then $Q_\chi(\infty)$ is continuous.  If $\chi_n \to \chi$ for some sequence so that $Q_{\chi_n}$ converges to a function $g$ uniformly on compact subsets of $\wti{\Omega}$, then by continuity, $g(\infty) = Q_\chi(\infty)$ and $\norm{g}_\infty \le 1$.  It follows by uniqueness of the minimizer that $g = Q_\chi$.  By Montel's Theorem, $\chi \mapsto Q_\chi$ is continuous.  Since $F_\chi(z) = Q_\chi(z)/Q_\chi(\infty)$ and $\norm{F_\chi}_\infty = 1/Q_\chi(\infty)$, we conclude continuity of $F_\chi$ and $\norm{F_\chi}_\infty$.

The converse follows from Theorem \ref{1.2} and $Q_\chi(\infty) = 1/\norm{F_\chi}_\infty$
\end{proof}

\begin{theorem} \lb{T3.2} Let $\fre \subset \bbC$ be a compact, PW and DCT set that is regular for potential theory. Then $n \mapsto F_{\chi_\fre^n}(z)$ and  $n \mapsto Q_{\chi_\fre^n}(z)$ are almost periodic uniformly for $z$ in compact subsets of $\wti{\Omega}$.  Moreover, $n \mapsto \norm{F_{\chi_\fre^n}}_\infty$ is a bounded almost periodic function.
\end{theorem}

\begin{proof} Almost periodicity of a function, $f$, on $\bbZ$ can be defined in terms of the family $f_m \equiv f(\cdot - m)$ lying in a compact family of functions.  Since $\pi_1(\Omega)^*$ is compact, $\{F_\chi\}_{\chi \in \pi_1(\Omega)^*}$ and $\{Q_\chi\}_{\chi \in \pi_1(\Omega)^*}$ are the required compact families.  Since $Q_\chi(\infty)$ is a continuous function, it takes its minimum value which is always non-zero.  Thus $Q_\chi(\infty)$ is bounded away from zero and thus, $\norm{F_\chi}_\infty = 1/Q_\chi(\infty)$ is bounded.
\end{proof}

We now turn to the proof of Theorem \ref{T1.5}.  The first two of four lemmas require neither almost periodicity nor canonical generator.  We'll focus on the dual maximizer, $Q_\chi$, given by \eqref{1.16}.

\begin{lemma} \lb{L3.3}  Let $\fre$ be a regular PW set.  Then $\chi \mapsto Q_\chi(\infty)$, the map from $\pi_1(\Omega)^*$ to $(0,1]$, is upper semicontinuous, i.e.,
\begin{equation}\label{3.1}
  \chi_j \to \chi \; \Rightarrow \; \limsup_{j \to \infty} Q_{\chi_j}(\infty) \le Q_\chi (\infty)
\end{equation}
\end{lemma}

\begin{proof} By Montel's  theorem, we can always pick a subsequence so that $Q_{\chi_{j_n}}(\infty) \to \limsup_{j \to \infty} Q_{\chi_j}(\infty)$ and so that $Q_{\chi_{j_n}}$ has a pointwise limit, $g$, on the universal cover which has $\norm{g}_\infty \le 1$ and for which the convergence is uniform on compact subsets of the universal cover.  Since $\chi_{j_n} \to \chi$, $g$ is a trial function for the dual Widom problem with character $\chi$.  Since $Q_\chi$ is a maximizer, $g(\infty) \le Q_\chi(\infty)$, i.e., \eqref{3.1} holds.
\end{proof}

\begin{lemma} \lb{L3.4} Let $\fre$ be a regular PW set.  If $\chi \mapsto Q_\chi(\infty)$ is continuous at $\chi = \bdone$ (i.e., we know that  $\chi_j \to \bdone \Rightarrow Q_\chi(\infty) \to 1$), then $\chi \mapsto Q_\chi(\infty)$  is continuous on $\pi_1(\Omega)^*$.
\end{lemma}

\begin{proof} Suppose $\chi_j \to c$.  Then $\chi_j/c \to \bdone$.  Since $Q_cQ_{\chi_j/c}$ is a trial function for the $\chi_j$ dual maximizer problem, we have that
\begin{equation}\label{3.2}
   Q_c(\infty)Q_{\chi_j/c}(\infty) \le Q_{\chi_j}(\infty)
\end{equation}
By hypothesis, $Q_{\chi_j/c}(\infty) \to 1$, so \eqref{3.2} implies that
\begin{equation}\label{3.3}
  Q_c(\infty) \le \liminf_{j \to \infty} Q_{\chi_j}(\infty).
\end{equation}
This and \eqref{3.1} imply that $Q_{\chi_j}(\infty) \to Q_c(\infty)$.
\end{proof}

\begin{lemma} \lb{L3.5}  Let $\fre$ be a regular PW set.  Suppose $n \mapsto \norm{F_n}_\infty$ is a bounded almost periodic function and that $\chi_\fre^{n_j} \to \bdone$.  Then $Q_{\chi_\fre^{n_j}} \to 1$.
\end{lemma}

\begin{proof}  By hypothesis, there exists a compact additive group $\bbK$ and a bounded continuous function, $B$, on $\bbK$ so that $\bbZ$ is a dense subgroup in $\bbK$ and $B(n) = \norm{F_n}_\infty$.  Let $A(\alpha) = B(\alpha)^{-1}$ which is also continuous on $\bbK$, bounded away from $0$ (and bounded above by 1) with
\begin{equation}\label{3.4}
  Q_{\chi_\fre^{n}}(\infty) = A(n)
\end{equation}
By passing to a subsequence, we can suppose that $n_j \to \alpha \in \bbK$ and that $Q_{\chi_\fre^{n_j}}(\infty)$ has a limit $q$.

Fix $n_s$.  By passing to a further subsequence, we can suppose that $Q_{\chi_\fre^{n_s-n_j}}$ has a limit, $g$, on the universal cover.  Since $\chi_\fre^{n_j} \to \bdone$, $g$ is a trial function for the $\chi_\fre^{n_s}$ problem so
\begin{equation}\label{3.5}
  Q_{\chi_\fre^{n_s}}(\infty) \ge g(\infty) = \lim_{n_j \to \infty} A(n_s-n_j) = A(n_s-\alpha)
\end{equation}
by the continuity of $A$.  Now take $n_s \to \infty$.  By definition of $q$, we have
\begin{equation*}
  q=\lim_{n_s \to \infty} Q_{\chi_\fre^{n_s}}(\infty) \ge \limsup_{n_s \to \infty} A(n_s-\alpha) = A(0) = 1
\end{equation*}
since $n_s \to \alpha$ and $A(0) = 1$ by \eqref{3.4}.  Thus $q \ge 1$.  Since $Q_\chi(\infty) \in (0,1]$, we conclude that $q=1$, i.e., $1$ is the only limit point of $Q_{\chi_\fre^{n_j}}(\infty)$ proving the lemma.
\end{proof}

\begin{lemma} \lb{L3.6} Let $\fre$ be a regular PW set. Suppose that $n \to \norm{F_n}_\infty$ is a bounded almost periodic function and that $\fre$ has a canonical generator.  Then $\chi \mapsto Q_\chi(\infty)$ is continuous at $\chi = \bdone$, i.e.,
\begin{equation}
\chi_j \to \bdone \; \Rightarrow \; \lim_{j \to \infty} Q_{\chi_j} (\infty) = 1
\end{equation}
\end{lemma}

\begin{proof}  $\pi_1(\Omega)^*$ is a compact, separable group, so metrizable.  Let $d$ be a metric on $\pi_1(\Omega)^*$ yielding the usual topology.  Since $\{\chi_\fre^m\}$ is dense, we can pick integers $m_j(\ell)$ for each $j$ and $\ell = 1, 2, \dots$ so that $d(\chi_j,\chi_\fre^{m_j(\ell)}) \le 2^{-\ell}$.

By Lemma \ref{L3.3}, we can pick $\ell_j \ge j$ so that
\begin{equation}\label{3.6}
  Q_{\chi_\fre^{m_j(\ell_j)}}(\infty) \le Q_{\chi_j}(\infty)+ 2^{-j}
\end{equation}
Let $k(j) = m_j(\ell_j)$.  Since $d(\bdone,\chi_\fre^{k(j)}) \le d(\bdone,\chi_j)+2^{-j}$, we see that $\chi_\fre^{k(j)} \to \bdone$, so by Lemma  \ref{3.5}, $Q_{\chi_\fre^{k(j)}}(\infty) \to 1$.  By \eqref{3.6}, we conclude that $\liminf Q_{\chi_j}(\infty) \ge 1$.  Since $Q_{\chi_j}(\infty) \in (0,1]$, we conclude that the limit is $1$.
\end{proof}

\begin{proof} [Proof of Theorem \ref{T1.5}]  By the hypothesis, Lemma \ref{L3.6} applies, so we conclude that $\chi \mapsto Q_\chi(\infty)$ is continuous at $\bdone$. By Lemma \ref{L3.4}, $\chi \mapsto Q_\chi(\infty)$ is continuous on all of $\pi_1(\Omega)^*$, so, by Theorem \ref{T1.2}, the set $\fre$ is DCT.
\end{proof}

\section{Limit Points of $M_n$ are Blaschke Products} \lb{s4}

In this section and the next, we consider the functions $M_n(z) = \left[B_\fre(z)/B_n(z)\right]^n$ of \eqref{1.21}.  Since $\fre \subset \fre_n$, we have that $G_n(z) \le G_\fre(z)$ so
\begin{equation}\label{4.1}
  |M_n(z)| \le 1
\end{equation}
$M_n(z)$ is analytic on the universal cover of $(\bbC \cup \{\infty\})\setminus \fre_n$.  Since the harmonic measures of components of $\fre_n$ are $j/n$, $B_n(z)^n$ is single valued analytic on $\bbC \setminus \fre_n$, so $M_n(z)$ has character $\chi_n \equiv \chi_\fre^n$ for curves in $\wti{\Omega}$ that avoid $\fre_n$.

In this section, we'll prove that limit points of $M_n$ (after removing some removable potential singular points) are Blaschke products analytic on $\wti{\Omega}$ and, in the next, that these Blaschke products are dual Widom maximizers.  This section will only require that $\fre \subset \bbR$ is regular for potential theory and obeys a PW condition while the next will also require the DCT condition.

$\bbR \setminus \fre$ is a disjoint union of bounded open components (plus two unbounded components), $K \in \calG$.  We'll call these the gaps and $\calG$ the set of gaps.  A \emph{gap collection} is a subset $\calG_0 \subset \calG$.  A \emph{gap set} is a gap collection, $\calG_0$, and for each $K_k \in \calG_0$ a point $x_k \in K_k$.  For any gap $K = (\beta-\alpha,\beta+\alpha)$, we define
\begin{equation*}
  K^{(\epsilon)} = (\beta-(1-\epsilon)\alpha,\beta+(1-\epsilon)\alpha)
\end{equation*}
so that $K^{(\epsilon)} \subset K$ and $|K^{(\epsilon)}| = (1-\epsilon)|K|$.

For any gap set, $S$, we define the associated Blaschke product
\begin{equation}\label{4.2}
  B_S(z) = \prod_{K_k \in \calG_0} B_\fre(z,x_k)
\end{equation}
Lifted to $\bbD$, each $B_\fre(z,x_k)$ is a product of elementary Blaschke factors and thus, so is the product in \eqref{4.2}.  It is known (\cite[Theorem 9.9.4]{BCA}) that such products either converge to $0$ uniformly on compacts, or else converge to an analytic function vanishing only at the individual zeros and, in the latter case, the product has $\lim_{r \uparrow 1} |B_S(\x(re^{i\theta}))| = 1$ for a.e. $\theta$ (\cite[Theorem 5.3.1]{HA}).  Since $\sum_{K \in \calG} \sup_{y \in K} G_\fre(\infty,y) < \infty$ by the PW condition, we see that the product in \eqref{4.2} converges to a non-zero value at $z=\infty$.  Thus $B_S(z)$ is an analytic function on $\wti{\Omega}$ which vanishes exactly at points $w$ with $\pi(w) \in \{x_j\}_{K_j \in \calG_0}$.  Moreover, for a.e. point $y \in \fre$,

\begin{equation}\label{4.3a}
  \lim_{\epsilon \downarrow 0} |B_S(y+i\epsilon)| = 1
\end{equation}

Recall (\cite[(b) following Theorem 1.1]{CSZ1}) that any Chebyshev polynomial, $T_n$, has at most one zero in any gap $K \in \calG$.  Our main result in this section is

\begin{theorem} \lb{T4.1} Let $n_j \to \infty$ so that for some gap set, $S$, we have that if $K_k \in \calG_0$, then for large $j$, $T_{n_j}(z)$ has a zero $z_j^{(k)}$ in $K_k$ which converges to $x_k$ as $j \to \infty$ and so that for any $K \in \calG \setminus \calG_0$, and for all $\epsilon >0$, $T_{n_j}(z)$ has no zero in $K^{(\epsilon)}$ for all large $j$.  Then, as $j \to \infty$, $M_{n_j}(z) \to B_S(z)$ uniformly on compact subsets of $\wti{\Omega} \setminus \{w\,|\,\pi(w) \in \{x_k\}\}$.
\end{theorem}

\begin{remarks} 1. The points $w$ with $\pi(w)=x_k$ for some $k$ are removable singular points for $B_S$.  In fact, it is easy to see that while $M_{n_j}(x_k+i0)$ and $M_{n_j}(x_k-i0)$ may be different, both values converge to $0$, so, in a certain sense, one has convergence on all of $\wti{\Omega}$.

2.  By Montel's Theorem and \eqref{4.1}, the functions $M_n$ lie in a compact set in the Fr\'echet topology of uniform convergence on compact subsets.  We can therefore make multiple demands and one might guess that, as in \cite{CSZ1}, we want to also demand that $\chi_{n_j}$ has a limit as does $[C(\fre_{n_j})/C(\fre)]^{n_j}$ and the $M_{n_j}$.  It turns out that the single condition on the limits of zeros will automatically imply these other objects converge.
\end{remarks}

We will prove this result by controlling convergence for $z$ near $\infty$ using

\begin{proposition} \lb{P4.2} Let $\Upsilon$ be a Riemann surface and $U_n$ open sets so that for any compact set $K \subset \Upsilon$, eventually, $K \subset U_n$.  Let $f_n$ be analytic functions on $U_n$ so that
\begin{equation}\label{4.3}
  \sup_n \sup_{z \in U_n} |f_n(z)| < \infty
\end{equation}
Let $f_\infty$ be analytic on $\Upsilon$ so that for some $z_0 \in \Upsilon$ and some neighborhood, $V$, of $z_0$, we have that
\begin{equation}\label{4.4}
  \lim_{n \to \infty} |f_n(z)| = |f_\infty(z)| \textrm{ for all } z \in V
\end{equation}
\begin{equation}\label{4.5}
  f_n(z_0) > 0, \qquad f_\infty(z_0) > 0
\end{equation}
\begin{equation}\label{4.6}
  z \in V \Rightarrow \forall n: f_n(z) \ne 0 \textrm{ and } f_\infty(z) \ne 0
\end{equation}
Then $f_n \to f$ uniformly on compact subsets of $\Upsilon$.
\end{proposition}

\begin{proof} By shrinking $V$, we can suppose that it is simply connected and $\overline{V}$ is compact.  By \eqref{4.5}/\eqref{4.6}, we can define $g_n(z) = \log f_n(z)$ uniquely if we demand that
\begin{equation}\label{4.7}
  \textrm{Im} g_n(z_0) = 0
\end{equation}
By \eqref{4.4}, $\textrm{Re}g_n \to \textrm{Re}g_\infty$ on $V$ so by the Cauchy--Riemann equations, $\nabla(\textrm{Im}g_n) \to \nabla(\textrm{Im}g_\infty)$.  By \eqref{4.7}, $\textrm{Im}g_n \to \textrm{Im}g_\infty$, so $f_n \to f_\infty$ on $V$.  By Vitali's Theorem (\cite[Section 6.2]{BCA}) and \eqref{4.3}, $f_n \to f_\infty$ uniformly on compacts.
\end{proof}

Thus instead of $M_n(z)$, we can look at
\begin{equation}\label{4.8}
  |M_n(z)| = \exp(-nh_n(z)), \qquad h_n(z) = G_\fre(z)-G_n(z)
\end{equation}
Let $d\rho_n$ be the potential theoretic equilibrium measure of $\fre_n$ (see \cite[Section 3.6--3.7]{HA} for background on potential theory).  Then

\begin{proposition} \lb{P4.3} One has that
\begin{equation}\label{4.9}
  h_n(z) = \int_{\bigcup_{K_j \in \calG} K_j} G_\fre(x,z) d\rho_n(x)
\end{equation}
\end{proposition}

\begin{remark} In \cite{CSZ1}, we proved the Totik--Widom bound \eqref{1.10} for PW sets, $\fre \subset \bbR$, by using this when $z=\infty$, i.e.,
\begin{equation*}
  h_n(\infty) = \int_{\bigcup_{K_j \in \calG} K_j} G_\fre(x) d\rho_n(x)
\end{equation*}
We proved this by thinking of $d\rho_n$ as harmonic measure at $\infty$, i.e., if $H$ is harmonic on $(\bbC \cup \{\infty\}) \setminus \fre_n$ with boundary values $H(x)$ on $\fre_n$, then
\begin{equation*}
  H(\infty) = \int_{\fre_n} H(x) d\rho_n(x)
\end{equation*}
If we wrote the analog of this for general $z$, we'd get
\begin{equation*}
   H(z) = \int_{\fre_n} H(x) d\rho_n(x,z)
\end{equation*}
varying the harmonic measure.  Instead we think of \eqref{4.9} with $G_\fre$ arising as the Green's function for solving Poisson's equation with zero boundary values on $\fre$ and $d\rho_n$ occurs as the Laplacian of $G_n$.
\end{remark}

\begin{proof} Both sides of \eqref{4.9} are continuous functions of $z \in \bbC \cup \{\infty\}$ (by regularity of $\fre$ and $\fre_n$) and both sides vanish on $\fre$.  Off $\fre$, they have the same distributional Laplacian, namely $d\rho_n \restriction (\fre_n \setminus \fre)$.  Thus the difference is harmonic on $(\bbC \cup \{\infty\}) \setminus \fre$, continuous on $\bbC \cup \{\infty\}$, vanishing on $\fre$ and bounded near $\infty$.  The boundedness means the difference is also harmonic at $\infty$ (\cite[Theorem 3.1.26]{HA}) and then the maximum principle implies that the difference is $0$.
\end{proof}

The final step in the proof of Theorem \ref{T4.1} involves the form as $n \to \infty$ of $d\rho_n \restriction K$ for $K \in \calG$. Recall that $\fre_n$ is a union of $n$ bands which are closures of the connected components of $T_n^{-1}[(-t_n,t_n)]$.  On each of these, as $x$ increases, $T_n$ is either strictly monotone increasing or strictly decreasing from $-t_n$ to $t_n$ or vice-versa. Recall also that each of the bands has $\rho_n$ measure exactly $1/n$ (see \cite[Thm. 2.3]{CSZ1}).  In \cite{CSZ1}, it is proven that each gap, $K$, contains all or part of a single band so that
\begin{equation}\label{4.10}
  n \rho_n(K) \le 1
\end{equation}
If there is $x_\infty \in K$ which is a limit as $j \to \infty$ of zeros, $x_{n_j}$ of $T_{n_j}$, then for $j$ large, $\fre_{n_j} \cap K$ is a complete band of exponentially small width so, in that case
\begin{equation}\label{4.11}
    n_j \rho_{n_j} \restriction K \to \delta_{x_\infty}
\end{equation}
weakly.  If for each $\epsilon$, there is a large $J_\epsilon$ so if $j \ge J_\epsilon$, then $T_{n_j}$ has no zero in $K^{(\epsilon)}$, then for all sufficiently large $j$, $\rho_{n_j}(K^{(\epsilon)}) = 0$.  Since $G_\fre$ vanishes at the edges of $K$ (and so $\sup_{x \in K \setminus K^{(\epsilon)}} G_\fre(x,z) \to 0$ as $\epsilon \downarrow 0$ uniformly as $z$ runs through compact sets), we conclude that
\begin{equation}\label{4.11A}
  n\int_K G_\fre(x,z) d\rho_n(x) \to \left\{
                                       \begin{array}{ll}
                                         G_\fre(x_\infty,z), & \textrm{if } K \in \calG_0 \\
                                         0, & \textrm{if } K \notin \calG_0
                                       \end{array}
                                     \right.
\end{equation}

By the PW condition, $\sum_{K \in \calG} \sup_{y \in K} G_\fre(z,y) < \infty$ uniformly in $z$ on compacts, we can go from pointwise limits in \eqref{4.11} to limits on sums.  We conclude that:

\begin{proposition} \lb{P4.4}  Under the hypotheses of Theorem \ref{T4.1}, uniformly for $z$ in compact subsets of $\Omega\setminus\{x_k\}_{K_k \in \calG_0}$, we have that
\begin{equation}\label{4.12}
  n\int_{\bigcup_{K_k \in \calG} K_k} G_\fre(x,z) d\rho_n(x) \to \sum_{K_k \in \calG_0} G_\fre(x_k,z)
\end{equation}
\end{proposition}

\begin{proof} [Proof of Theorem \ref{4.1}] By \eqref{4.8}, \eqref{4.9} and \eqref{4.12},
\begin{equation}\label{4.13}
  \lim_{n_j \to \infty} |M_{n_j}(z)| = \prod_{K_k \in \calG_0} |B_\fre(z,x_k)| = |B_S(z)|
\end{equation}
That $M_{n_j} \to B_S$ then follows from Proposition \ref{P4.2}.
\end{proof}

\section{Blaschke Products are Dual Widom Maximizers} \lb{s5}

Given the setup of Theorem \ref{T4.1}, the function $B_S(z)$ is character automorphic with some character $\beta$.  In this section, we'll prove that $B_S$ is a dual Widom maximizer for character $\beta$.  One can deduce this from results of Volberg--Yuditskii \cite[Lemma 6.4]{VY}.  Instead, we'll follow an approach of Eichinger--Yuditskii \cite{EichY} (who study an Ahlfors problem rather than a dual Widom problem) that relies on results of Sodin--Yuditskii \cite{SY}.

A basic technique of Sodin--Yuditskii is to consider the space, $\calH_\alpha$, of all functions on $\wti{\Omega}$ which are in $H^2(\bbD)$ when moved to $\bbD$ and which are character automorphic with character $\alpha \in \pi_1(\Omega)^*$.  $\calH_\alpha$ is a family of functions on $\wti{\Omega}$ which is a reproducing kernel Hilbert space (\cite[Problems 4--11 of Section 3.3]{RA}) under the inner product of $H^2$.  In particular, there is a function $K^\alpha \in \calH_\alpha$ so that for all $f \in \calH_\alpha$
\begin{equation}\label{5.1}
  f(\infty) = \jap{K^\alpha,f}
\end{equation}
Note: Our inner products are linear in the second factor and anti-linear in the first as in \cite{RA}.

We will prove

\begin{theorem} \lb{T5.1} For any gap set, $S$, if $B_S$ is the associated Blaschke product and $\beta$ its character, then $B_S$ is a dual Widom maximizer for $\beta$, i.e.,
\begin{equation}\label{5.2}
  \norm{B_S}_\infty = 1
\end{equation}
and if $f \in H^\infty(\Omega,\beta)$ with $\norm{f}_\infty \le 1$, then
\begin{equation}\label{5.3}
     |f(\infty)| \le B_S(\infty)
\end{equation}
\end{theorem}

\eqref{5.2} is, of course, true for any (convergent) Blaschke product.  We prove \eqref{5.3} by proving two facts:

(1) For any character, $\gamma$, and $f \in H^\infty(\Omega,\beta)$ with $\norm{f}_\infty \le 1$, one has that
\begin{equation}\label{5.4}
  |f(\infty)|^2 \le \frac{K^{\gamma\beta}(\infty)}{K^\gamma(\infty)}
\end{equation}

(2) There exists at least one $\alpha_0$ with
\begin{equation}\label{5.5}
  |B_S(\infty)|^2 = \frac{K^{\alpha_0\beta}(\infty)}{K^{\alpha_0}(\infty)}
\end{equation}

\begin{lemma} \lb{L5.2} \eqref{5.4} holds.
\end{lemma}

\begin{proof}  Since $f \in H^\infty(\Omega,\beta)$ and $K^\gamma \in \calH_\gamma$, we have that $fK^\gamma \in \calH_{\gamma\beta}$.  Thus
\begin{align}
  |f(\infty)K^\gamma(\infty)|^2 &= |\jap{K^{\gamma\beta},fK^\gamma}|^2 \nonumber \\
                                &\le \norm{fK^\gamma}_2^2 \norm{K^{\gamma\beta}}_2^2 \lb{5.6} \\
                                &\le \norm{K^\gamma}_2^2 \norm{K^{\gamma\beta}}_2^2 \lb{5.7} \\
                                &= \jap{K^\gamma,K^\gamma} \jap{K^{\gamma\beta},K^{\gamma\beta}} \nonumber \\
                                &= K^\gamma(\infty) K^{\gamma\beta}(\infty) \lb{5.8}
\end{align}
which is \eqref{5.4} since $K^\gamma(\infty)>0$.  In the above, \eqref{5.6} is the Schwarz inequality, \eqref{5.7} uses $\norm{f}_\infty \le 1$ and \eqref{5.8} is \eqref{5.1}.
\end{proof}

For step 2, we need a deep result of Sodin--Yuditskii.  For each gap $K \in \calG$, we define $C_K$ to be two copies glued together at the ends, i.e., we take two copies $\{(y,+),(y,-) \, |\,y \in \overline{K}\}$ and for $y \in \partial K$ (two points), we set $(y,+)=(y,-)$ so $C_K$ is topologically a circle.  According to Sodin--Yuditskii \cite{SY}, there is a map, $\frA$, the Abel map, from $\prod_{K \in \calG} C_K$ to the character group, so that, in particular, the inner part of $K^{\frA(y,\sigma)}$ is $B_S$ where $S$ is the gap set with
\begin{equation*}
\calG_0 = \{K \,|\,(y_K,\sigma_K) \textrm{ has } \sigma_K=+ \textrm{ and } y_K \in K\}
\end{equation*}
(i.e., $y_K \notin \partial K$) and for $K \in \calG_0$, the point in $K$ is $y_K$.

In particular, if $S$ is given and $(y,\sigma)=\{(y_K,\sigma_K)\}_{K\in\calG}$ is picked so that for $K_k \in \calG_0$, we have that $(y_{K_k},\sigma_{K_k})=(x_k,+)$ (and for $K \notin \calG_0$, $(y_K,\sigma_K)$ is arbitrary in $C_K$), then the inner factor of $K^{\frA(y,\sigma)}$ is divisible by $B_S$, i.e., if $\alpha_1 = \frA(y,\sigma)$, then $K^{\alpha_1}/B_S$ is in $\calH_{\alpha_0}$ where $\alpha_0=\alpha_1 \beta^{-1}$.  If $g \in \calH_{\alpha_0}$, then because multiplication by $B_S$ is an isometry on $H^2$, we have that
\begin{align}
  \jap{K^{\alpha_0\beta}B_S^{-1},g} &= \jap{K^{\alpha_0\beta},B_S g} \nonumber \\
                                    &= B_S(\infty)g(\infty) \lb{5.9} \\
                                    &= B_S(\infty) \jap{K^{\alpha_0},g} \lb{5.10} \\
                                    &= \jap{\overline{B_S(\infty)}K^{\alpha_0},g} \lb{5.11}
\end{align}

Since $g$ is arbitrary in $\calH_{\alpha_0}$ and both $K^{\alpha_0}$ and $K^{\alpha_0\beta}B_S^{-1}$ lie in $\calH_{\alpha_0}$, we conclude that
\begin{equation}\label{5.12}
  K^{\alpha_0\beta}(z)B_S(z)^{-1} = \overline{B_S(\infty)}K^{\alpha_0}(z)
\end{equation}
Evaluating at $z=\infty$, we find that

\begin{lemma} \lb{L5.3} \eqref{5.5} holds for $\alpha_0 = \alpha_1\beta^{-1}$ where $\alpha_1$ is the image under the Abel map of data $\{(y_K,\sigma_K)\}_{K \in \calG}$ which has $(y_{K_k},\sigma_{K_k})=(x_k,+)$ if $K_k \in \calG_0$.
\end{lemma}

\begin{proof} [Proof of Theorem \ref{T5.1}] By Lemmas \ref{L5.2} and \ref{L5.3}, if $g \in H^\infty(\Omega,\beta)$ with $\norm{g}_\infty \le 1$, then
\begin{equation}\label{5.13}
  |g(\infty)|^2 \le \frac{K^{\alpha_0\beta}(\infty)}{K^{\alpha_0}(\infty)} = |B_S(\infty)|^2
\end{equation}
Thus, if $g(\infty)>0$, we have that
\begin{equation}\label{5.14}
  0 < g(\infty) \le B_S(\infty)
\end{equation}
so $B_S$ is a dual Widom maximizer.
\end{proof}
\section{Proof of the Main Theorem} \lb{s6}

In this section, we'll prove Theorem \ref{T1.3}.

\begin{proposition} \lb{P6.1} Under the hypotheses of Theorem \ref{T4.1}, we have that $L_{n_j}(z)$ (given by \eqref{1.20}) converges uniformly on compact subsets of $\wti{\Omega}$ to the Widom minimizer for the character, $\beta$, of $B_S$.
\end{proposition}

\begin{remark} $M_n$ only converge away from the $\{x_k\}_{K_k \in \calG_0}$ because the $M_n$'s aren't analytic on $\wti{\Omega}$ but only on those points whose images under $\x$ aren't in $\fre_n$.  But $L_n$ is analytic on all of $\wti{\Omega}$ so we can hope for convergence at the $x_k$'s too.  Indeed, the $x_k$'s are limit points of zeros and the Widom minimizers vanish at those points.
\end{remark}

\begin{proof} We have that $M_{n_j}(\infty)=\left[C(\fre)/C(\fre_{n_j})\right]^{n_j}$, so by Theorem \ref{T4.1},
\begin{equation}\label{6.1}
  B_S(\infty) = \lim_{j \to \infty} \left[C(\fre)/C(\fre_{n_j})\right]^{n_j}
\end{equation}
Thus, if $H_n$ is given by \eqref{1.23}, then
\begin{equation}\label{6.2}
  H_{n_j}(z) \to B_S(z)/B_S(\infty)
\end{equation}
for $z$ near $\infty$ (in fact on compact subsets of $\wti{\Omega} \setminus \{w \,|\, \pi(w) \in \{x_k\} \}$).

Since $B_S$ is the dual Widom maximizer for $\beta$, $B_S(z)/B_S(\infty)$ is $F_\beta$, the Widom minimizer for $\beta$.  By the first equation in \eqref{1.23}, we get that $L_{n_j}(z)$ converges to $F_\beta(z)$ for $z$ near $\infty$.

By the Totik--Widom bound, $\norm{L_{n_j}}_{\infty}$ are uniformly bounded, so by Vitali's Theorem, $L_{n_j}$ converges to $F_\beta$ uniformly on compact subsets of $\wti{\Omega}$.
\end{proof}

\begin{proposition} \lb{P6.2}  Under the hypotheses of Theorem \ref{T4.1}, we have that
\begin{equation}\label{6.3}
  \lim_{j \to \infty} \norm{L_{n_j}}_\infty = 2 \norm{F_\beta}_\infty
\end{equation}
\end{proposition}

\begin{proof} Since $\log|L_{n_j}(z)|$ is harmonic on $\Omega$ away from those zeros of $T_{n_j}$ in the gaps where it goes to $-\infty$, its maximum occurs at limit points on $\fre$.  Since $|B_\fre(x)| = 1$ for $x \in \fre$, we conclude that
\begin{equation} \lb{6.4}
  \norm{L_{n_j}}_\infty = \frac{t_{n_j}}{ C(\fre)^{n_j}} = \frac{2C(\fre_{n_j})^{n_j}}{C(\fre)^{n_j}}
\end{equation}
by \eqref{1.6a}

By \eqref{6.1}, we conclude that
\begin{equation}\label{6.5}
  \lim_{j \to \infty} \norm{L_{n_j}}_\infty  = 2\left[B_S(\infty)\right]^{-1}
\end{equation}
and by \eqref{1.16}, noting that $Q_\beta = B_S$,
\begin{equation}\label{6.6}
  \left[B_S(\infty)\right]^{-1} = \norm{F_\beta}_\infty
\end{equation}
proving \eqref{6.3}.
\end{proof}

\begin{proof} [Proof of Theorem \ref{T1.3}]  By Theorem \ref{T3.2}, we have the required almost periodicity of $F_n(z)$ and $\norm{F_n}_\infty$.  By continuity of $\norm{F_\chi}_\infty$ and the Totik--Widom bound, the functions on the left of \eqref{1.18} lie in a compact set, so if the limit is not zero, by passing to suitable subsequences, we can find one whose limit is zero for which the hypotheses of Theorem \ref{T4.1} hold.  But then the limit is zero by Proposition \ref{P6.1}.  We conclude that \eqref{1.18} holds.

Again, by continuity of $\norm{F_\chi}_\infty$ and the Totik--Widom bound, the numbers on the left side of \eqref{1.18A} are bounded above and away from zero, so if \eqref{1.18A} fails we can find a subsequence for which the limit is not $2$ and for which the hypotheses of Theorem \ref{T4.1} hold.  This violates Proposition \ref{P6.2} so we conclude that \eqref{1.18A} holds.
\end{proof}


\end{document}